\documentclass[11pt]{amsart}

\usepackage{amsmath}
\usepackage{amsthm}
\usepackage{amssymb}
\usepackage{mathrsfs}
\usepackage{comment}
\usepackage{color}
\usepackage[colorlinks,citecolor=blue,urlcolor=black,linkcolor=black]{hyperref}

\newtheorem{theorem}{Theorem}[section]
\newtheorem{claim}[theorem]{Claim}

\newtheorem{lemma}[theorem]{Lemma}

\newtheorem{corollary}[theorem]{Corollary}

\newtheorem*{theorem*}{Theorem}

\theoremstyle{definition}
\newtheorem{definition}[theorem]{Definition}

\newtheorem{question}[theorem]{Question}

\theoremstyle{remark}

\newcount\skewfactor
\def\mathunderaccent#1#2 {\let\theaccent#1\skewfactor#2
\mathpalette\putaccentunder}
\def\putaccentunder#1#2{\oalign{$#1#2$\crcr\hidewidth
\vbox to.2ex{\hbox{$#1\skew\skewfactor\theaccent{}$}\vss}\hidewidth}}

\def\smallbox#1{\leavevmode\thinspace\hbox{\vrule\vtop{\vbox
   {\hrule\kern1pt\hbox{\vphantom{\tt/}\thinspace{\tt#1}\thinspace}}
   \kern1pt\hrule}\vrule}\thinspace}

\def\qedref#1{$\qed_{\reforiginal{#1}}$}

\title{Quadruples and cubes}
\author{Shimon Garti}
\address{Einstein Institute of Mathematics,
 The Hebrew University of Jerusalem,
 Jerusalem 9190401, Israel}
\email{shimon.garty@mail.huji.ac.il}

\thanks{Research supported by ISF grant no. 2320/23}

\begin{document}
\let\labeloriginal\label
\let\reforiginal\ref
\def\ref#1{\reforiginal{#1}}
\def\label#1{\labeloriginal{#1}}

\begin{abstract}
We prove, in \textsf{ZFC}, that if $\lambda$ is an uncountable cardinal then the negative arrow relation $\left( \begin{smallmatrix} \lambda^{++} \\ \lambda^+ \\ \lambda \end{smallmatrix} \right) \nrightarrow \left( \begin{smallmatrix} \lambda^{++} \\ \lambda^+ \\ \lambda \end{smallmatrix} \right)$ holds.
If $\lambda=\aleph_0$ then the positive relation $\left( \begin{smallmatrix} \lambda^{++} \\ \lambda^+ \\ \lambda \end{smallmatrix} \right) \rightarrow \left( \begin{smallmatrix} \lambda^{++} \\ \lambda^+ \\ \lambda \end{smallmatrix} \right)$ has consistency strength of at least one Woodin cardinal.
The corresponding terraced relation for quadruples fails at every infinite cardinal $\lambda$, in \textsf{ZFC}.
Finally, we show that a wondrous ideal implies a positive polarized relation at a successor and a double successor.
However, from the results on cubes we conclude that there are no Galvin-related wondrous ideals over two consecutive cardinals.
\end{abstract}

\maketitle

\newpage

\section{Introduction}

Positive cube polarized relations are much more challenging than positive relations on pairs.
In this paper we discuss cube relations, and the main result gives a negative answer, in \textsf{ZFC}, to a question about such relations.
Let us commence with the formal definition of these polarized relations.

\begin{definition}
  \label{defpolcube} The cube polarized relations $\left( \begin{smallmatrix} \alpha \\ \beta \\ \gamma \end{smallmatrix} \right) \rightarrow \left( \begin{smallmatrix} \varepsilon \\ \zeta \\ \eta \end{smallmatrix} \right)$ says that for every coloring $d:\alpha\times\beta\times\gamma\rightarrow 2$ one can find $A\subseteq\alpha,B\subseteq\beta$ and $C\subseteq\gamma$ such that ${\rm otp}(A)=\varepsilon,{\rm otp}(B)=\zeta,{\rm otp}(C)=\eta$ and $d\upharpoonright(A\times B\times C)$ is constant.
\end{definition}

It is clear that if either pair in the cube fails to satisfy the pertinent relation then adding a third coordinate will not rectify the problem. Thus a positive relation at every pair of cardinals (there are three such pairs) is a necessary condition for obtaining a positive cube relation.
Notwithstanding, positive cube relations are quite rare, even in cases where the relevant positive relation holds at every pair out of the three cardinals in the domain of the coloring.

One may wonder why is it so, and the following observation explains this issue to some extent.
As we shall see, negative relations at pairs \emph{with many colors} give rise to a negative relation for cubes \emph{with two colors only}.
Since negative relations with many colors are easily obtained, one can show that negative cube relations hold in many cases, even under circumstances in which positive relations hold with two colors at every pair.
This idea will be exploited in the next section.
Addressing a question from \cite{MR4880649} we shall prove the following:

\begin{theorem}
  \label{utheorem} If $\lambda$ is uncountable then $\left( \begin{smallmatrix} \lambda^{++} \\ \lambda^+ \\ \lambda \end{smallmatrix} \right) \nrightarrow \left( \begin{smallmatrix} \lambda^{++} \\ \lambda^+ \\ \lambda \end{smallmatrix} \right)$.
\end{theorem}

The above relation is called the $\lambda$-terraced cube relation.
It follows from this theorem that the corresponding $\lambda$-terraced relation for quadruples fails even if $\lambda=\aleph_0$.
But for cube relations we need the assumption that $\lambda$ is uncountable.
One may wonder what happens if $\lambda=\aleph_0$, namely whether $\left( \begin{smallmatrix} \aleph_2 \\ \aleph_1 \\ \aleph_0 \end{smallmatrix} \right) \rightarrow \left( \begin{smallmatrix} \aleph_2 \\ \aleph_1 \\ \aleph_0 \end{smallmatrix} \right)$ is consistent with \textsf{ZFC}.
We shall see that this positive relation has a considerable consistency strength, and it requires at least one Woodin cardinal, if consistent at all.

These results have some bearing on polarized relations at pairs of cardinals.
An interesting open problem is whether the strong polarized relation is consistent at some pair of successor and double successor cardinals.
We shall prove that if there exists a wondrous ideal over $\kappa^+$ then $\binom{\kappa^{++}}{\kappa^+}\rightarrow\binom{\kappa^{++}}{\kappa^+}_\kappa$.
From the negative cube relation we can deduce now that there are no wondrous ideals over $\kappa^+$ and $\kappa^{++}$ simultaneously, under an additional assumption about the connection between these ideals.

Cube (and quadruple) relations in the current paper are \emph{strong}, in the sense that the size of the domain of the colorings and the required size of the monochromatic product are identical.
We refer the reader to \cite{MR1297180} and \cite{MR4381747} for interesting theorems about $n$-dimensional polarized relations in which the monochromatic product is strictly less than the cardinality of the coloring.

Our notation is mostly standard, and follows \cite{MR795592}.
An excellent monograph, containing a chapter about polarized relations, is \cite{MR3075383}.
The rest of the paper is arranged in two sections.
The first section deals with cubes (and also quadruples).
The second section is dedicated to classical polarized relations in which the small component is a successor cardinal.
In this section we introduce the concept of wondrous ideals and show how to get positive polarized relations from such ideals.

\newpage

\section{Terraced cubes}

Let $\lambda$ be an infinite cardinal.
The terraced cube at $\lambda$ is the relation $\left( \begin{smallmatrix} \lambda^{++} \\ \lambda^+ \\ \lambda \end{smallmatrix} \right) \rightarrow \left( \begin{smallmatrix} \lambda^{++} \\ \lambda^+ \\ \lambda \end{smallmatrix} \right)_\theta$.
If $\theta=2$, i.e. the coloring has two colors only, then we omit the subscript.

It has been shown in \cite[Theorem 2.1]{MR4880649} that the $\lambda$-terraced cube relation fails if $2^\lambda\leq\lambda^{++}$.
A natural problem is whether this relation is consistent.
In models of \textsf{ZF} one can prove such relations, under the relevant assumptions, see \cite{MR4101445}.
However, in the framework of \textsf{ZFC} it is not clear whether such a positive relation is consistent, see \cite[Question 2.2]{MR4880649}.
Our goal, in this section, is to supply a negative answer.

We commence with a general lemma that transfers questions about cubes into problems about pairs:

\begin{lemma}
  \label{lempairs} Suppose that $\nu>\mu>\lambda\geq\theta$.
  Assume that there are $d:\nu\times\mu\rightarrow\theta$ and $e:\mu\times\lambda\rightarrow\theta$ satisfying the following properties:
  \begin{enumerate}
    \item [$(a)$] For every $A\in[\nu]^{\nu'}$ and $B\in[\mu]^{\mu'}$ there exists $\beta\in{B}$ so that $|\{d(\alpha,\beta)\mid\alpha\in{A}\}|=\theta$, where $\nu'\leq\nu$ and $\mu'\leq\mu$.
    \item [$(b)$] For every $B\in[\mu]^{\mu'}$ and every $C\in[\lambda]^{\lambda'}$ there exists $\beta\in{B}$ so that $|\{e(\beta,\gamma)\mid\gamma\in{C}\}|=\theta$, where $\mu'\leq\mu$ and $\lambda'\leq\lambda$.
  \end{enumerate}
  Then there is a coloring $c:\nu\times\mu\times\lambda\rightarrow\{0,1\}$ witnessing $\left( \begin{smallmatrix} \nu \\ \mu \\ \lambda \end{smallmatrix} \right) \nrightarrow \left( \begin{smallmatrix} \nu' \\ \mu' \\ \lambda' \end{smallmatrix} \right)$.
\end{lemma}

\par\noindent\emph{Proof}. \newline
Given $\alpha\in\nu,\beta\in\mu,\gamma\in\lambda$ let $c(\alpha,\beta,\gamma)=0$ iff $d(\alpha,\beta)\leq e(\beta,\gamma)$.
Thus $c(\alpha,\beta,\gamma)=1$ iff $d(\alpha,\beta)>e(\beta,\gamma)$.
Suppose that $A\in[\nu]^{\nu'}, B\in[\mu]^{\mu'}$ and $C\in[\lambda]^{\lambda'}$.
Choose $\beta_1\in{B}$ for which $|\{d(\alpha,\beta_1)\mid\alpha\in{A}\}|=\theta$.
Choose $\gamma_1\in{C}$ and $\alpha_1\in{A}$ so that $d(\alpha_1,\beta_1)>e(\beta_1,\gamma_1)$.
This is possible since $\{d(\alpha,\beta_1)\mid\alpha\in{A}\}$ is unbounded in $\theta$.
By definition, $c(\alpha_1,\beta_1,\gamma_1)=1$.

Now choose $\beta_0\in{B}$ for which $|\{e(\beta_0,\gamma)\mid\gamma\in{C}\}|=\theta$.
Choose $\alpha_0\in{A}$ and $\gamma_0\in{C}$ so that $d(\alpha_0,\beta_0)\leq e(\beta_0,\gamma_0)$.
This is possible since the set $\{e(\beta_0,\gamma)\mid\gamma\in{C}\}$ is unbounded in $\theta$.
By definition, $c(\alpha_0,\beta_0,\gamma_0)=0$.
The proof is accomplished.

\hfill \qedref{lempairs}

The important feature of the lemma is the shift from $\theta$-many colors in $d$ and $e$ to $\{0,1\}$ in $c$.
In order to apply this lemma in the specific case of terraced cube relations, we need two additional claims.
The first one is rather trivial.

\begin{lemma}
  \label{lemtrivial} Let $\lambda$ be an infinite cardinal.
  There exists a coloring $e:\lambda^+\times\lambda\rightarrow\lambda$ such that for every $B\in[\lambda^+]^{\lambda^+}$ and $C\in[\lambda]^\lambda$ one can find $\beta\in{B}$ for which $|\{e(\beta,\gamma)\mid\gamma\in{C}\}|=\lambda$.
\end{lemma}

\par\noindent\emph{Proof}. \newline
Define $e(\beta,\gamma)=\gamma$.

\hfill \qedref{lemtrivial}

The second claim is a bit more involved.

\begin{definition}
  \label{defcolorful} Let $\lambda$ be an infinite cardinal.
  A coloring $d:\lambda^{++}\times\lambda^+\rightarrow\lambda$ is $\lambda$-colorful if for every $B\in[\lambda^+]^{\lambda^+}$ and $A\in[\lambda^{++}]^{\lambda^{++}}$ there exists $\beta\in{B}$ so that $|\{d(\alpha,\beta)\mid\alpha\in{A}\}|=\lambda$.
\end{definition}

We shall prove that a colorful function exists under the assumption that there is a sad family (the acronym sad stands for stationary almost disjoint).
Let $\lambda$ be an infinite cardinal.
A family $\mathcal{S}=(S_\delta\mid\delta\in\mu)$ of subsets of $\lambda^+$ is \emph{sad} if and only if every $S_\alpha$ is a stationary subset of $\lambda^+$ and $S_\alpha\cap S_\beta$ is bounded in $\lambda^+$ whenever $\alpha<\beta<\mu$.

\begin{claim}
  \label{clmsadcolorful} Let $\lambda$ be an infinite cardinal.
  If there exists a $\lambda^+$-sad family $(S_\delta\mid\delta\in\lambda^{++})$ then there is a colorful $d$ for $\lambda$.
\end{claim}

\par\noindent\emph{Proof}. \newline
Fix an almost disjoint family $(A_\alpha\mid\alpha\in\lambda^{++})$ of subsets of $\lambda^+$, each of which is of cardinality $\lambda^+$.
Let $(S_\delta\mid\delta\in\lambda^{++})$ be a sad family in $\lambda^+$.
For every $\alpha\in\lambda^{++}$ we define a regressive function $h_\alpha:\lambda^+\rightarrow{A_\alpha}$, which is surjective, by the following procedure.
First, we fix $M_\alpha\subseteq\lambda^{++}$ for every $\alpha\in\lambda^{++}$ such that $|M_\alpha|=\lambda^+$ and if $\alpha<\beta<\lambda^{++}$ then $M_\alpha\cap M_\beta=\varnothing$.
Second, for every $\xi\in{A_\alpha}$ we choose $\delta(\xi,\alpha)\in{M_\alpha}$ such that if $\zeta\neq\xi$ then $\delta(\alpha,\zeta)\neq\delta(\alpha,\xi)$.
Finally, we let $h''S_{\delta(\alpha,\xi)}=\{\xi\}$.
We may assume, without loss of generality, that the domain of $h_\alpha$ is $\lambda^+$ (if not, just add arbitrary values to $h_\alpha$ where needed).
For every $\beta\in\lambda^+$ choose a one-to-one function $g_\beta:\beta\rightarrow\lambda$.
We define:
$$
d(\alpha,\beta)=g_\beta(h_\alpha(\beta)).
$$
Let us show that $d$ is $\lambda$-colorful.
Fix $B\in[\lambda^+]^{\lambda^+}$ and $A\in[\lambda^{++}]^{\lambda^{++}}$.
Choose a subset of $A$ of size $\lambda$, say $\{\alpha_i\mid i\in\lambda\}\subseteq{A}$.
For every $i<j<\lambda$ and each $\xi\in A_{\alpha_i}\cap A_{\alpha_j}$ let $W_{ij\xi}=S_{\delta(\alpha_i,\xi)}\cap S_{\delta(\alpha_j,\xi)}$.
By our assumptions, $|W_{ij\xi}|\leq\lambda$.
Letting $W_{ij}=\bigcup\{W_{ij\xi}\mid\xi\in A_{\alpha_i}\cap A_{\alpha_j}\}$ we see that $|W_{ij}|\leq\lambda$ for every $i<j<\lambda$.

Define $B'=\bigcup\{W_{ij}\mid i<j<\lambda\}$.
It follows that $|B'|\leq\lambda$, and hence $B-B'\neq\varnothing$.
Pick any ordinal $\beta\in B-B'$.
By definition, $\beta\notin W_{ij}$ whenever $i<j<\lambda$.
We claim that $h_{\alpha_i}(\beta)\neq h_{\alpha_j}(\beta)$.
Indeed, if $h_{\alpha_i}(\beta)=h_{\alpha_j}(\beta)=\xi$ then $\xi\in A_{\alpha_i}\cap A_{\alpha_j}$ and hence $\beta\in S_{\delta(\alpha_i,\xi)}\cap S_{\delta(\alpha_j,\xi)}=W_{ij\xi}$.
However, $W_{ij\xi}\subseteq B'$ while $\beta\notin B'$, a contradiction.

Thus $h_{\alpha_i}(\beta)\neq h_{\alpha_j}(\beta)$ for every $i<j<\lambda$.
Let $R=\{h_{\alpha_i}(\beta)\mid i\in\lambda\}$.
Since each $h_\alpha$ is regressive we see that $R\subseteq\beta$.
By the above arguments, $|R|=\lambda$.
Recall that $g_\beta$ is one-to-one, and hence $|g_\beta''R|=\lambda$.
To accomplish the proof notice that $|\{d(\alpha,\beta)\mid\alpha\in{A}\}|\geq|\{d(\alpha_i,\beta)\mid i\in\lambda\}|=\lambda$, so we are done.

\hfill \qedref{clmsadcolorful}

One may wonder, at this point, whether there are $\lambda^+$-sad families of size $\lambda^{++}$.
By an observation of Hayut,\footnote{Articulated for $\omega_1$, \cite{hayut}, but works at every successor cardinal, as proved in the next claim.} this is closely related to the saturation of the non-stationary ideal of $\lambda^+$.
Recall that ${\rm NS}_{\lambda^+}$ is \emph{saturated} iff for every collection $(T_\varepsilon\mid\varepsilon\in\lambda^{++})$ of stationary subsets of $\lambda^+$ there are $\varepsilon<\zeta<\lambda^{++}$ for which $T_\varepsilon\cap T_\zeta$ is stationary.
The existence of a $\lambda^+$-sad family of size $\lambda^{++}$ seems to be a strong version of the non-saturation of ${\rm NS}_{\lambda^+}$.
However, it turns out that these two statements are equivalent.

\begin{claim}
  \label{clmsadandns} Let $\lambda$ be an infinite cardinal.
  Assume that ${\rm NS}_{\lambda^+}$ is not saturated.
  Then there exists a $\lambda^+$-sad family $(T_\alpha\mid\alpha\in\lambda^{++})$.
\end{claim}

\par\noindent\emph{Proof}. \newline
Let $(S_\alpha\mid\alpha\in\lambda^{++})$ be a witness for the non-saturation of ${\rm NS}_{\lambda^+}$.
For every $\alpha<\beta<\lambda^{++}$ choose a club $C_{\alpha\beta}\subseteq\lambda^+$ disjoint from $S_\alpha\cap S_\beta$.
For every $\beta\in\lambda^{++}$ let $E_\beta$ be the diagonal intersection of all the clubs in $\{C_{\alpha\beta}\mid\alpha\in\beta\}$, upon re-enumerating the elements of this family by an enumeration whose order-type is at most $\lambda^+$.

For every $\alpha\in\lambda^{++}$ let $T_\alpha=E_\alpha\cap S_\alpha$.
Thus each $T_\alpha$ is a stationary subset of $\lambda^+$, and we claim that $(T_\alpha\mid\alpha\in\lambda^{++})$ is a $\lambda^+$-sad family.
Indeed, if $\alpha<\beta<\lambda^{++}$ then $C_{\alpha\beta}$ is disjoint from $T_\alpha\cap T_\beta$ since $T_\alpha\cap T_\beta\subseteq S_\alpha\cap S_\beta$.
Moreover, $E_\beta\subseteq^* C_{\alpha\beta}$ which means that $T_\alpha\cap T_\beta\subseteq E_\beta-C_{\alpha\beta}$ and hence $|T_\alpha\cap T_\beta|\leq\lambda$, so we are done.

\hfill \qedref{clmsadandns}

Let us indicate that the argument appearing within the proof hinges on the fact that the size of the family is $\lambda^{++}$.
We can prove now the following:

\begin{theorem}
  \label{thmnegterraced} Let $\lambda$ be an uncountable cardinal.
  Then $\left( \begin{smallmatrix} \lambda^{++} \\ \lambda^+ \\ \lambda \end{smallmatrix} \right) \nrightarrow \left( \begin{smallmatrix} \lambda^{++} \\ \lambda^+ \\ \lambda \end{smallmatrix} \right)$.
\end{theorem}

\par\noindent\emph{Proof}. \newline
From \cite{MR1363421} we know that ${\rm NS}_{\lambda^+}$ is not saturated.
By Claim \ref{clmsadandns}, there is a $\lambda^+$-sad family $(T_\alpha\mid\alpha\in\lambda^{++})$.
Hence there is a colorful $d:\lambda^{++}\times\lambda^+\rightarrow\lambda$ as dictated in Claim \ref{clmsadcolorful}.
Let $e:\lambda^+\times\lambda\rightarrow\lambda$ be as guaranteed by Lemma \ref{lemtrivial}.
We define a coloring $c:\lambda^{++}\times\lambda^+\times\lambda\rightarrow\{0,1\}$.
Given $\alpha\in\lambda^{++},\beta\in\lambda^+$ and $\gamma\in\lambda$ we let $c(\alpha,\beta,\gamma)=0$ iff $d(\alpha,\beta)\leq e(\beta,\gamma)$, thus $c(\alpha,\beta,\gamma)=1$ iff $d(\alpha,\beta)>e(\beta,\gamma)$.

Suppose that $A\in[\lambda^{++}]^{\lambda^{++}}, B\in[\lambda^+]^{\lambda^+}$ and $C\in[\lambda]^\lambda$.
Fix $\beta\in{B}$ so that $|\{e(\beta,\gamma)\mid\gamma\in{C}\}|=\lambda$.
Choose $\alpha\in{A}$ and let $\varepsilon=d(\alpha,\beta)$.
Now pick $\gamma\in{C}$ such that $e(\beta,\gamma)\geq\varepsilon=d(\alpha,\beta)$.
It follows that $c(\alpha,\beta,\gamma)=0$.
For the second color, choose $\beta\in{B}$ so that $|\{d(\alpha,\beta)\mid\alpha\in{A}\}|=\lambda$.
Choose $\gamma\in{C}$ and let $\varepsilon=e(\beta,\gamma)$.
Choose $\alpha\in{A}$ so that $d(\alpha,\beta)>\varepsilon$.
It follows from the definition of $c$ that $c(\alpha,\beta,\gamma)=1$, thereby proving the theorem.

\hfill \qedref{thmnegterraced}

The above argument applies, basically, to the case of $\lambda=\aleph_0$ as well.
However, the non-stationary ideal over $\aleph_1$ can be saturated, so in order to obtain the negative cube relation one has to assume that ${\rm NS}_{\aleph_1}$ is not saturated.
Shelah proved that the assertion that ${\rm NS}_{\aleph_1}$ is saturated has consistency strength of one Woodin cardinal.\footnote{The proof can be extracted from \cite{MR1623206}.}

\begin{corollary}
  \label{cornegcubealeph1} The consistency strength of the statement $\left( \begin{smallmatrix} \aleph_2 \\ \aleph_1 \\ \aleph_0 \end{smallmatrix} \right) \rightarrow \left( \begin{smallmatrix} \aleph_2 \\ \aleph_1 \\ \aleph_0 \end{smallmatrix} \right)$ is at least one Woodin cardinal.
\end{corollary}

In this specific case, \cite[Theorem 2.1]{MR4880649} is still meaningful, and it implies that in order to force a positive relation, one has to force $2^{\aleph_0}\geq\aleph_3$.
The scenario in which ${\rm NS}_{\omega_1}$ is saturated and $2^{\aleph_0}\geq\aleph_3$ is unclear,\footnote{See \cite[Question 39]{MR2768692}.} and consequently we do not know whether the $\aleph_0$-terraced relation is consistent with \textsf{ZFC}.
We indicate, however, that under \textsf{AD} the positive relation $\left( \begin{smallmatrix} \aleph_2 \\ \aleph_1 \\ \aleph_0 \end{smallmatrix} \right) \rightarrow \left( \begin{smallmatrix} \aleph_2 \\ \aleph_1 \\ \aleph_0 \end{smallmatrix} \right)$ issues, as shown in \cite{MR4101445}.
Moreover, this is the only positive strong cube relation that holds under \textsf{AD} when all three cardinals are regular.

In the opposite direction, one can conclude that a negative polarized relation always holds (in \textsf{ZFC}) for quadruples, since such a relation contains a cube of three uncountable cardinals even when $\lambda=\aleph_0$.

\begin{corollary}
  \label{corquadruples} Let $\lambda$ be an infinite cardinal.
  Then $\left( \begin{smallmatrix} \lambda^{+3} \\ \lambda^{++} \\ \lambda^+ \\ \lambda \end{smallmatrix} \right) \nrightarrow \left( \begin{smallmatrix} \lambda^{+3} \\ \lambda^{++} \\ \lambda^+ \\ \lambda \end{smallmatrix} \right)$.
\end{corollary}

\newpage

\section{wondrous ideals}

The most investigated instance of the polarized partition relation is the case of $\kappa$ and $\kappa^+$ for some infinite cardinal $\kappa$.
In these cases, $\binom{\kappa^+}{\kappa}\rightarrow\binom{\kappa^+}{\kappa}_2$ cannot be established in \textsf{ZFC}.
Indeed, by \cite{MR0202613} if $2^\kappa=\kappa^+$ then $\binom{\kappa^+}{\kappa}\nrightarrow\binom{\kappa^+}{\kappa}_2$.
Thus in order to obtain the strong relation at $\kappa$ and its successor, one has to force $2^\kappa>\kappa^+$.

The first positive result in this direction belongs to Laver\footnote{See \cite{MR371652}. The statement follows from a similar result that Laver proved in that paper.} who proved that $\binom{\aleph_1}{\aleph_0}\rightarrow\binom{\aleph_1}{\aleph_0}_2$ follows from $\mathsf{MA}_{\omega_1}$.
This result remained isolated for many years, partly because it is difficult to generalize Martin's axiom to higher cardinals.
However, Laver's result follows from the weaker assumption $\mathfrak{s}>\aleph_1$, as shown in \cite[Claim 2.4]{MR2927607}.
Moving from Martin's axiom to the splitting number $\mathfrak{s}$, one obtains the strong relation $\binom{\kappa^+}{\kappa}\rightarrow\binom{\kappa^+}{\kappa}_2$ where $\kappa$ is a sufficiently large cardinal by forcing $\mathfrak{s}_\kappa>\kappa^+$, see \cite[Claim 1.2]{MR3201820}.
By \cite{MR1251349}, $\mathfrak{s}_\kappa>\kappa^+$ means that $\kappa$ is at least weakly compact.
Actually, if $\kappa$ is sufficiently large then one can force $\mathfrak{s}_\kappa>\kappa^+$.
Thus in the category of large cardinals, the strong relation is forceable in many cases.\footnote{The possible consistency of this relation at small large cardinals (e.g., strongly inaccessible but not weakly compact) is an open problem.}

Another class of cardinals for which this problem is well-understood is the class of strong limit singular cardinals.
Observe that if $\mu$ is a strong limit singular cardinal then $2^\mu>\mu^+$ necessitates large cardinals.
In fact, the exact consistency strength of $2^\mu>\mu^+$ at strong limit singular cardinals was determined by Gitik, see \cite{MR1007865}.
Nevertheless, in many cases one can force $\binom{\mu^+}{\mu}\rightarrow\binom{\mu^+}{\mu}_2$ where $\mu$ is a strong limit singular cardinal, see \cite{MR2987137} and \cite{MR3509813}.

The last category is, therefore, successor cardinals (as the small component in the polarized relation).
The question is whether $\binom{\kappa^{++}}{\kappa^+}\rightarrow\binom{\kappa^{++}}{\kappa^+}_2$ is consistent for some infinite cardinal $\kappa$.
We indicate that in \textsf{ZF} such a relation is consistent, e.g. $\binom{\aleph_2}{\aleph_1}\rightarrow\binom{\aleph_2}{\aleph_1}_{\aleph_0}$ holds under \textsf{AD}, see \cite{MR4101445}.
However, it is unknown whether such a relation is consistent with \textsf{ZFC}.

It should be emphasized that a positive relation in this category, even at the first relevant case (i.e., $\aleph_1$ and $\aleph_2$) has some consistency strength.
By a result of Galvin (see \cite[Claim 3.1]{MR4531656}), if there is a Kurepa tree then $\binom{\omega_2}{\omega_1}\nrightarrow\binom{2}{\omega_1}_{\aleph_0}$.
Hence the positive relation $\binom{\aleph_2}{\aleph_1}\rightarrow\binom{\aleph_2}{\aleph_1}_{\aleph_0}$ requires at least an inaccessible cardinal, and actually it requires much more (see \cite{MR1024901}).
The almost strong relation $\binom{\kappa^{++}}{\kappa^+}\rightarrow\binom{\tau}{\kappa^+}_\kappa$ for every $\tau\in\kappa^{++}$ (where $\kappa$ is a regular cardinal) can be forced from a huge cardinal above $\kappa$, see \cite{MR3610266}.
Huge cardinals are instrumental here since by collapsing the predecessors of such a cardinal to $\kappa$ one produces a highly saturated ideal over $\kappa^+$.
Following this line, we define the concept of a wondrous ideal and we show that the existence of such an ideal over $\kappa^+$ implies $\binom{\kappa^{++}}{\kappa^+}\rightarrow\binom{\kappa^{++}}{\kappa^+}_\kappa$.

The definition of a wondrous ideal brings together two features of strong saturation.
The first one is \emph{denseness}, and the second is \emph{the enhanced Galvin property}.
Let us define these concepts at $\aleph_1$ (the generalization to $\kappa^+$ is straightforward).

A (uniform) ideal $\mathcal{I}$ over $\aleph_1$ is called \emph{dense} iff there exists a collection $\{y_\eta\mid\eta\in\omega_1\}\subseteq\mathcal{I}^+$ such that for every $A\in\mathcal{I}^+$ there exists $\eta\in\omega_1$ for which $y_\eta\subseteq_{\mathcal{I}}A$.
From very large cardinals one can force such an ideal, as proved by Woodin (see \cite{MR2768692}).
In the usual constructions of such an ideal one obtains a normal ideal (in particular, the ideal is $\aleph_1$-complete).
Let us indicate that denseness implies $(\aleph_2,\aleph_2,\aleph_0)$-saturation, the latter has been forced for ideals over $\aleph_1$ by Laver in \cite{MR673792}, starting from a huge cardinal in the ground model.

The second property that we need will be called \emph{the enhanced Galvin property}.
Let $\kappa$ be a regular uncountable cardinal, and let $\mathscr{F}$ be a normal filter over $\kappa$.
Galvin proved that if $\kappa$ is strongly regular\footnote{A cardinal $\kappa$ is strongly regular iff $\kappa=\kappa^{<\kappa}$.} then every family $\{A_\alpha\mid\alpha\in\kappa^+\}\subseteq\mathscr{F}$ admits a subfamily $\{A_{\alpha_i}\mid i\in\kappa\}$ such that $\bigcap\{A_{\alpha_i}\mid i\in\kappa\}\in\mathscr{F}$.
In particular, if $2^\kappa=\kappa^+$ then $\kappa^+$ is strongly regular and then every normal filter over $\kappa^+$ satisfies Galvin's property.
The enhanced Galvin property requires a bit more: every family $\{A_\alpha\mid\alpha\in\kappa^+\}\subseteq\mathscr{F}$ admits a subfamily $\{A_{\alpha_i}\mid i\in\kappa^+\}$ such that $\bigcap\{A_{\alpha_i}\mid i\in\kappa^+\}\in\mathscr{F}$.

\begin{definition}
  \label{defwondrous} Wondrous ideals. \\
  Let $\mathcal{I}$ be an $\aleph_1$-complete ideal over $\aleph_1$, and let $\mathscr{F}_{\mathcal{I}}$ be the dual filter.
  We shall say that $\mathcal{I}$ is wondrous iff it is $\aleph_1$-dense and the filter $\mathscr{F}_{\mathcal{I}}$ satisfies the enhanced Galvin property.
\end{definition}
The above definition is articulated at the level of $\aleph_1$ but generalizes verbatim to $\kappa^+$ whenever $\kappa$ is an infinite cardinal.
The main result of this section is that wondrous ideals imply the strong polarized relation at successors and double successors.

Though we do not know how to force the existence of wondrous ideals, we can prove (from the results of the previous section) that in some sense there are no such ideals over two consecutive cardinals.
For if $\kappa$ is regular and uncountable, $\mathcal{I}$ is wondrous over $\kappa, \mathcal{J}$ is wondrous over $\kappa^+$ and these ideals are connected in a simple way to be described anon, then the cube relation at $\kappa$ follows.
Since the cube relation fails, we conclude that there is a meaningful limitation on wondrous ideals.
Let us commence with colorings of pairs.

\begin{theorem}
  \label{thmwondstrong} If there exists a wondrous ideal over $\kappa^+$ then $\binom{\kappa^{++}}{\kappa^+}\rightarrow\binom{\kappa^{++}}{\kappa^+}_{\kappa}$.
\end{theorem}

\par\noindent\emph{Proof}. \newline
Let $\mathcal{I}$ be a wondrous ideal over $\kappa^+$.
Let $c:\kappa^{++}\times\kappa^+\rightarrow\kappa$ be a coloring.
For every $\alpha\in\kappa^{++}$ let $\chi(\alpha)\in\kappa$ be an ordinal for which $x_\alpha=\{\beta\in\kappa^+\mid c(\alpha,\beta)=\chi(\alpha)\}\in\mathcal{I}^+$.
We can choose such $\chi(\alpha)$ by virtue of $\kappa^+$-completeness of $\mathcal{I}$.
Let $x=\{x_\alpha\mid\alpha\in\kappa^{++}\}$.

Let $\{y_\eta\mid\eta\in\kappa^+\}\subseteq\mathcal{I}^+$ be a witness to the density of $\mathcal{I}$.
For every $\alpha\in\kappa^{++}$ choose $\eta(\alpha)\in\kappa^+$ so that $y_{\eta(\alpha)}\subseteq_{\mathcal{I}}x_\alpha$.
By the pigeonhole principle there are $A_0\in[\kappa^{++}]^{\kappa^{++}}$ and $\eta\in\kappa^+$ such that $\eta(\alpha)=\eta$ for every $\alpha\in{A_0}$.
By a similar argument, there are $A_1\in[A_0]^{\kappa^{++}}$ and $\gamma\in\kappa$ such that $\chi(\alpha)=\gamma$ whenever $\alpha\in{A_1}$.

For each $\alpha\in{A_1}$ let $Z_\alpha\in\mathcal{I}$ be such that $y_\eta-Z_\alpha\subseteq{x_\alpha}$.
Let $W_\alpha=\omega_1-Z_\alpha$ for every $\alpha\in{A_1}$, and observe that $W_\alpha\in\mathscr{F}_{\mathcal{I}}$ for each such $\alpha$, and hence $\{W_\alpha\mid\alpha\in{A_1}\}\subseteq\mathscr{F}_{\mathcal{I}}$.
By the enhanced Galvin property, there is a set $A\in[A_1]^{\kappa^{++}}$ so that $W=\bigcap\{W_\alpha\mid\alpha\in{A}\}\in\mathscr{F}_{\mathcal{I}}$.
Let $Z=\omega_1-W$, so $Z\in\mathcal{I}$.
By the above considerations, $Z_\alpha\subseteq{Z}$ for every $\alpha\in{A}$.

Let $B=y_\eta-Z$, so $B\in\mathcal{I}^+$ and in particular $|B|=\aleph_1$.
We claim that $c''(A\times{B})=\{\gamma\}$.
To see this, fix $\alpha\in{A}$ and $\beta\in{B}$.
So $\beta\in y_\eta-Z$ and since $Z_\alpha\subseteq{Z}$ we know that $\beta\in y_\eta-Z_\alpha$.
Recall that $\alpha\in{A}$ and hence $y_\eta=y_{\eta(\alpha)}$ and therefore $\beta\in y_{\eta(\alpha)}-Z_\alpha$.
This means that $\beta\in{x_\alpha}$, so by the definition of $x_\alpha$ and the fact that $\alpha\in{A}\subseteq{A_0}$ one infers that $c(\alpha,\beta)=\gamma$, as required.

\hfill \qedref{thmwondstrong}

Let $\kappa$ be a regular and uncountable cardinal.
Let $\mathcal{I}$ be a $\kappa$-complete ideal over $\kappa$ and let $\mathcal{J}$ be a $\kappa^+$-complete ideal over $\kappa^+$.
We shall say that $\mathcal{I}$ is $\mathcal{J}$-Galvin iff for every collection $\{W_\alpha\mid\alpha\in\kappa^+\}\subseteq\mathscr{F}_{\mathcal{I}}$ there exists $A\in\mathcal{J}^+$ such that $\bigcap\{W_\alpha\mid\alpha\in{A}\}\in\mathscr{F}_{\mathcal{I}}$.

\begin{lemma}
  \label{lemjgalvin} Let $\kappa$ be regular and uncountable and let $u$ be a subset of $\kappa^{++}$ of size $\kappa^{++}$.
  Let $\mathcal{I}$ be a wondrous ideal over $\kappa$, let $\mathcal{J}$ be a wondrous ideal over $\kappa^+$, and assume that $\mathcal{I}$ is $\mathcal{J}$-Galvin.
  If $\{B^\delta\mid\delta\in{u}\}\subseteq\mathcal{I}^+$ then there are $b\in\mathcal{I}^+$ and $e\in[u]^{\kappa^{++}}$ such that $b\subseteq B^\delta$ for every $\delta\in{e}$.
\end{lemma}

\par\noindent\emph{Proof}. \newline
By the denseness of $\mathcal{I}$, there is $y\in\mathcal{I}^+$ and $u'\in[u]^{\kappa^{++}}$ such that $y\subseteq_{\mathcal{I}}B^\delta$ for every $\delta\in u'$.
We may assume, without loss of generality, that $u'=u$.
Our goal is to find a fixed $Z\in\mathcal{I}$ and a set $e\in[u]^{\kappa^{++}}$ so that $b=y-Z\subseteq B^\delta$ for every $\delta\in{e}$.

Let $(M_\xi\mid\xi\in\kappa^{++})$ be a partition of $u$ into sets of cardinality $\kappa^+$, so $\zeta<\xi<\kappa^{++}$ implies $M_\zeta\cap M_\xi=\varnothing$.
Fix $\xi\in\kappa^{++}$ and enumerate the elements of $\{B^\delta\mid\delta\in M_\xi\}$ by $\{B^\xi_\alpha\mid\alpha\in\kappa^+\}$.
Since $\mathcal{I}$ is $\mathcal{J}$-Galvin, there are $v_\xi\in\mathcal{J}^+$ and $Z_\xi\in\mathcal{I}$ such that $y-Z_\xi\subseteq B^\xi_\alpha$ for every $\alpha\in v_\xi$.
We choose such sets for every $\xi\in\kappa^{++}$.
By the denseness of $\mathcal{J}$ applied to $\{v_\xi\mid\xi\in\kappa^{++}\}\subseteq\mathcal{J}^+$ there are $a\in[\kappa^{++}]^{\kappa^{++}}$ and $v\in\mathcal{J}^+$ so that $v\subseteq_{\mathcal{J}}v_\xi$ for every $\xi\in{a}$.
Without loss of generality, $a=u$.
By the enhanced Galvin property of $\mathcal{J}$ (and shrinking $v$ if needed) we may assume that $v\subseteq v_\xi$ for every $\xi\in{u}$.

If $\zeta,\xi\in{u}$ and $\zeta\neq\xi$ then possibly $Z_\zeta\neq Z_\xi$, but for every $\alpha\in v_\zeta\cap v_\xi$ both $y-Z_\zeta$ and $y-Z_\xi$ are subsets of $B^\zeta_\alpha$ and $B^\xi_\alpha$.
Letting $\xi_0$ be the first element of $u$ and $Z=Z_{\xi_0}$, we see that $y-Z\subseteq B^\xi_\alpha$ for every $\xi\in{u}$ and $\alpha\in{v}$.
Let $b=y-Z$, so $b\in\mathcal{I}^+$ and it satisfies the requirements of the lemma.

\hfill \qedref{lemjgalvin}

Equipped with this lemma, we can prove the following.

\begin{theorem}
  \label{thmconsecutivewondrous} Let $\kappa$ be a regular and uncountable cardinal.
  There are no wondrous ideals $\mathcal{I}$ over $\kappa$ and $\mathcal{J}$ over $\kappa^+$ such that $\mathcal{I}$ is $\mathcal{J}$-Galvin.
\end{theorem}

\par\noindent\emph{Proof}. \newline
By way of contradiction assume that such ideals are given.
We shall show that the positive relation $\left( \begin{smallmatrix} \kappa^{++} \\ \kappa^+ \\ \kappa \end{smallmatrix} \right) \rightarrow \left( \begin{smallmatrix} \kappa^{++} \\ \kappa^+ \\ \kappa \end{smallmatrix} \right)$ follows, contradicting Theorem \ref{thmnegterraced}.
Assume, therefore, that $c:\kappa^{++}\times\kappa^+\times\kappa\rightarrow\{0,1\}$ is a coloring.
For every $\delta\in\kappa^{++}$ let $c_\delta=c\upharpoonright\{\delta\}\times\kappa^+\times\kappa$.

We repeat the argument within the proof of the previous theorem with respect to $c_\delta$, but now we obtain a monochromatic product in which the large component belongs to $\mathcal{J}^+$ and the small component is an element of $\mathcal{I}^+$.
Fix $\delta\in\kappa^{++}$.
For every $\alpha\in\kappa^{+}$ let $\chi(\alpha)\in\kappa$ be an ordinal for which $x^\delta_\alpha=\{\beta\in\kappa\mid c_\delta(\alpha,\beta)=\chi(\alpha)\}\in\mathcal{I}^+$.
Let $x^\delta=\{x^\delta_\alpha\mid\alpha\in\kappa^{+}\}$.

Fix a dense family $(y_\eta\mid\eta\in\kappa)\subseteq\mathcal{I}^+$.
For every $\alpha\in\kappa^+$ choose $\eta(\alpha)\in\kappa$ so that $y_{\eta(\alpha)}\subseteq_{\mathcal{I}}x^\delta_\alpha$.
By the completeness of $\mathcal{J}$, one can choose $A_0^\delta\in\mathcal{J}^+$ and a fixed $\eta\in\kappa$ so that $\eta(\alpha)=\eta$ whenever $\alpha\in A^\delta_0$.
Shrink $A_0^\delta$ again, if needed, by choosing $A_1^\delta\subseteq A_0^\delta$ and $\gamma\in\{0,1\}$ such that $A_1^\delta\in\mathcal{J}^+$ and $\chi(\alpha)=\gamma$ whenever $\alpha\in A_1^\delta$.
For every $\alpha\in A_1^\delta$ pick $Z_\alpha^\delta\in\mathcal{I}$ such that $y_\eta-Z_\alpha^\delta\subseteq x^\delta_\alpha$.
Let $W_\alpha^\delta=\kappa-Z^\delta_\alpha$, so $W^\delta_\alpha\in\mathscr{F}_{\mathcal{I}}$ for every $\alpha\in A_1^\delta$.
Apply the assumption that $\mathcal{I}$ is $\mathcal{J}$-Galvin to obtain $A^\delta\subseteq A^\delta_1$ so that $A^\delta\in\mathcal{J}^+$ and $W^\delta=\bigcap\{W^\delta_\alpha\mid\alpha\in A^\delta\}\in\mathscr{F}_{\mathcal{I}}$.
In other words, the set $Z^\delta=\kappa-W^\delta$ is an element of $\mathcal{I}$.
Letting $B^\delta=y_\eta-Z^\delta$ we see that $B^\delta\in\mathcal{I}^+$ and $c_\delta''(A^\delta\times B^\delta)=\{\gamma\}$, as wanted.

We obtain $A^\delta$ and $B^\delta$ for every $\delta\in\kappa^{++}$.
Having the family $\{A^\delta\mid\delta\in\kappa^{++}\}\subseteq\mathcal{J}^+$ at hand, we find an element $v\in\mathcal{J}^+$ and a set $u\in[\kappa^{++}]^{\kappa^{++}}$ such that $v\subseteq A^\delta$ whenever $\delta\in{u}$.
We apply, first, the denseness of $\mathcal{J}$ to find $v'\in\mathcal{J}^+$ and $u'\in[\kappa^{++}]^{\kappa^{++}}$ such that $v'\subseteq_{\mathcal{J}}A^\delta$ for every $\delta\in{u'}$.
Then we employ the enhanced Galvin property to obtain $Z\in\mathcal{J}$ and $u\in[u']^{\kappa^{++}}$ such that $v'-Z\subseteq A^\delta$ for every $\delta\in{u}$.
Letting $v=v'-Z$ we see that $v\in\mathcal{J}^+$, as sought.
Apply Lemma \ref{lemjgalvin} to the family $\{B^\delta\mid\delta\in\kappa^{++}\}$.
Let $b\in\mathcal{I}^+$ and $e\in[\kappa^{++}]^{\kappa^{++}}$ be such that $b\subseteq B^\delta$ for every $\delta\in{e}$.
By re-enumerating we may assume that $b\subseteq B^\delta$ for every $\delta\in\kappa^{++}$.

Consider the product $u\times{v}\times{b}$.
The cardinality of this product is $\kappa^{++}\times\kappa^+\times\kappa$.
If $\delta\in{u}, \alpha\in{v}$ and $\beta\in{b}$ then $\beta\in B^\delta$ since $b\subseteq B^\delta$ and $\alpha\in A^\delta$ since $v\subseteq A^\delta$.
we conclude, therefore, that $c(\delta,\alpha,\beta)=c_\delta(\alpha,\beta)=\gamma$, thus $c''(u\times{v}\times{b})=\{\gamma\}$.
As indicated at the beginning of the proof, this is impossible.

\hfill \qedref{thmconsecutivewondrous}

A major problem remains open:

\begin{question}
  \label{qdouble} Is it consistent that $\binom{\kappa^{++}}{\kappa^+}\rightarrow\binom{\kappa^{++}}{\kappa^+}_2$ holds for some infinite cardinal $\kappa$
\end{question}

If one proves that wondrous ideals (consistently) exist then a positive answer to this problem ensues.
Each one of the two properties of wondrous ideals is forceable by itself; the difficulty is to force them together.
Though wondrous ideals over $\kappa^+$ yield a positive relation of this form, we indicate that such ideals give more than two colors.
It is fairly possible that a weaker saturation property implies a positive answer to this question.

\newpage

\bibliographystyle{alpha}
\bibliography{arlist}

\end{document}